\documentclass[a4paper, 10pt, reqno, english, oneside]{amsart}  % amsart or article
\usepackage[T1]{fontenc}

\usepackage[
    a4paper,
    left=3.5cm,
    right=3.5cm,
    top=4.0cm,
    bottom=4.0cm
]{geometry}

\usepackage{times}

\usepackage{color}
\usepackage[utf8]{inputenc}
\usepackage[dvipsnames]{xcolor}
\usepackage[colorlinks=true,urlcolor={YellowOrange},linkcolor={YellowOrange},citecolor={YellowOrange}]{hyperref}  
\usepackage{float}
\usepackage{amsfonts,amssymb,amsmath}
\usepackage{amsthm}    % can only be used if option "nospthms" is there
\usepackage{url}
\usepackage{paralist}
\usepackage{tikz}
    \usetikzlibrary{decorations.markings,arrows,shapes.callouts}
\usepackage{rotating} 
\usepackage{verbatim}
\usepackage{tikz-cd}
\usepackage{mathtools}
\usepackage{dsfont}
\usepackage{amsrefs}
\usepackage{thmtools}
\usepackage{thm-restate}

\newtheorem{theorem}{Theorem}[section]

\newtheorem{lemma}[theorem]{Lemma}

\newtheorem{corollary}[theorem]{Corollary} 
\newtheorem{problem}[theorem]{Problem}
\theoremstyle{definition}
\newtheorem{definition}[theorem]{Definition}
\theoremstyle{remark}
\newtheorem{example}[theorem]{Example}

\newcommand{\R}{\mathbb{R}}
\newcommand{\C}{\mathbb{C}}
\newcommand{\FF}{\mathbb{F}}

\newcommand{\Z}{\mathbb{Z}}

\newcommand{\F}{\mathcal{F}}
\newcommand{\G}{\mathcal{G}}
\newcommand{\I}{\mathcal{I}}
\newcommand{\J}{\mathcal{J}}

\newcommand{\catTop}{\text{Top}}
\newcommand{\CC}{\mathcal{C}}
\newcommand{\re}{\text{Re}}

\def\keywords{\xdef\@thefnmark{}\@footnotetext}

\DeclareMathOperator{\im}{\mathrm{im}}
\DeclareMathOperator{\sk}{\mathrm{sk}}

\DeclareMathOperator{\link}{\mathrm{lk}}

\DeclareMathOperator{\rank}{\mathrm{rk}}
\DeclareMathOperator{\conv}{\mathrm{conv}}

\DeclareMathOperator{\linspan}{\mathrm{span}}

\DeclareMathOperator{\hocolim}{\mathrm{hocolim}}

\newcommand{\stiefel}[2]{V_{#1}(\mathbb{F}^{#2})}

\tikzcdset{row sep/my_size/.initial=-2mm}
\tikzcdset{row sep/my_size_small/.initial=-3mm}
\newcommand{\xDownarrow}[1]{%
  {\left\Downarrow\vbox  to #1{}\right.\kern-\nulldelimiterspace}
}

\usepackage{tikz}
\usetikzlibrary{arrows.meta,calc}

\begin{document}

\title[On colorful generalizations of the Goodman--Pollack transversal problem]{On colorful generalizations of the Goodman--Pollack transversal problem}

%----------------------------------------------
% ALL INFORMATION ABOUT THE AUTHORS STARTS HERE

\author[Sadovek]{Nikola Sadovek} 
\address{Center for Systems Biology Dresden, Dresden, Germany \newline
Faculty of Mathematics, Technische Universit\"at Dresden, Dresden, Germany \newline
Max Planck Institute of Molecular Cell Biology and Genetics, Dresden, Germany}
\email{sadovek@mpi-cbg.de}

%\thanks{\hspace{-4mm}\textit{2020 Mathematics Subject Classification.} 52A35. \\ 
%\textit{Keywords.} Geometric transversals. Matroids. Helly-type theorems. }

% ALL INFORMATION ABOUT THE AUTHORS ENDS HERE
%--------------------------------------------

\begin{abstract}
    We establish a colorful and, more generally, matroidal solution to the problem of Goodman and Pollack on the existence of an $\FF$-affine $k$-dimensional transversal to a family of convex sets in $\mathbb{F}^d$, where $0 \le k \le d - 1$ is an integer and $\mathbb{F} \in \{\mathbb{R}, \mathbb{C}\}$.
    Our results unify several classical and recent theorems. In the case $k=0$, we recover the colorful Helly theorem of Lov\'asz, together with a matroidal extension due to Kalai and Meshulam. In the opposite extreme case $k=d-1$, we obtain Holmsen’s colorful and matroidal generalization of the Goodman--Pollack--Wenger theorem. Additionally, we extend the recent noncolorful solution of the Goodman-Pollack problem by McGinnis and the author.
    As the main application, we obtain a matroidal and colorful Dol'nikov-type transversal theorem.
    Our methods are topological. We introduce matroidal joins, defined as homotopy colimits of diagrams over the face poset of matroidal complexes, and derive estimates on their connectivity. The proof additionally relies on adaptations of nonexistence results for equivariant maps from Stiefel manifolds to spheres.
\end{abstract}

\maketitle

%\setcounter{tocdepth}{1}
%\tableofcontents

%\subjclass{\noindent\textit{2020 Mathematics Subject Classification.} \\}
%\keywords{\noindent\textit{Key words and phrases.}}

\section{Introduction}
\label{sec:intro}

\subsection*{History and context}

For integers $0 \le k < d$ and a family $\F$ of convex sets in $\R^d$ a \emph{$k$-transversal} is a $k$-dimensional affine subspace of $\R^d$ that intersects every member of $\F$. Typically, in geometric transversal theory, it is of interest to find necessary and sufficient conditions that guarantee the existence of a $k$-transversal; for extensive overviews of the subject consult the surveys of Goodman, Pollack, and Wenger~\cite{GoodmanGeometric1993} and Holmsen and Wenger~\cite{HolmsenChapter}.

A first in a long history of geometric transversal results is Helly's theorem~\cite{helly1923mengen}, which provides a necessary and sufficient condition for the existence of a $0$-transversal, that is, a point common to all sets in the family. Its \emph{colorful} generalization of Lov\'asz that first appeared in a paper of B\'ar\'any is one of the cornerstones of \mbox{discrete geometry.}

\begin{theorem}[{Lov\'asz~\cite{lovasz1974colorful} and B\'ar\'any~\cite{barany1982generalization}}] \label{thm:col-helly}
    Let $\F$ be a finite family of convex sets in $\R^d$ and let $\F = \CC_1 \cup \dots \cup \CC_{d+1}$ be a coloring into $d+1$ nonempty color classes. Suppose that for every choice $F_1 \in \CC_1, \dots, F_{d+1} \in \CC_{d+1}$ we have $F_1 \cap \dots \cap F_{d+1} \neq \emptyset$. Then, there exists a color $i$ such that $\bigcap \CC_i \neq \emptyset$.
\end{theorem}

The result, known as the \emph{colorful} Helly theorem, reduces to Helly’s original result by setting $\CC_1 = \dots = \CC_{d+1}$. The dual version of Theorem~\ref{thm:col-helly} due to B\'ar\'any~\cite{barany1982generalization}, known as the \emph{colorful Carath\'eodory theorem}, has variety of applications in discrete geometry \cites{barany1995caratheodory, sarkaria1992tverberg}.

\begin{figure}
    \centering

\begin{tikzpicture}[
    scale=0.8,
    poly/.style={draw=black, thick},
    plane/.style={draw=black},
    dashedhull/.style={dashed},
    point/.style={circle, fill=black, inner sep=1.5pt},
    >=Latex
]

% =====================================================
% LEFT: four polytopes
% =====================================================

% Top left polytope
\coordinate (A1) at (-4,0.9);
\draw[poly] ($(A1)+(-0.4,-0.3)$) %left
    -- ($(A1)+(0.3,-0.4)$)
    -- ($(A1)+(0.5,0.2)$) %right
    -- ($(A1)+(0.0,0.5)$)
    -- ($(A1)+(0.3,-0.4)$)
    -- ($(A1)+(-0.4,-0.3)$)
    -- ($(A1)+(0.0,0.5)$);
\coordinate (A3) at ($(A1)+(-0.4,-0.3)$); % left
\coordinate (A4) at ($(A1)+(0.5,0.2)$); % right

% Bottom left polytope
\coordinate (A2) at (-4.5,-1.2);
\draw[poly] ($(A2)+(-0.5,-0.2)$)
    -- ($(A2)+(0.2,-0.5)$) % bottom
    -- ($(A2)+(-0.1,0.4)$)
    -- ($(A2)+(0.2,-0.5)$)
    -- ($(A2)+(0.5,0.1)$)
    -- ($(A2)+(-0.1,0.4)$) %top
    -- cycle;
\coordinate (A5) at ($(A2)+(-0.1,0.4)$); % top
\coordinate (A6) at ($(A2)+(0.2,-0.5)$) ; % bottom

% Top right polytope
%\coordinate (B1) at (-1.0,1.0);
%\draw[poly] ($(B1)+(-0.5,-0.2)$)
%    -- ($(B1)+(0.2,-0.6)$) %bottom
%    -- ($(B1)+(0.6,0.0)$)
%    -- ($(B1)+(0.0,0.5)$) %
%    -- cycle;

% Top right polytope
\coordinate (B1) at (-0.9,1.0);
\draw[poly] ($(B1)+(0,-0.5)$) %bottom
    -- ($(B1)+(0.3,-0.1)$)
    -- ($(B1)+(0.3,0.4)$)
    -- ($(B1)+(-0.3,0.4)$) %top
    -- ($(B1)+(-0.6,0)$)
    -- ($(B1)+(-0.6,-0.5)$)
    -- cycle;
\coordinate (B5) at ($(B1)+(-0.3,0.4)$); % top
\coordinate (B6) at ($(B1)+(0,-0.5)$) ; % bottom
\draw[thick] (B1) -- (B6);
\draw[thick] (B1) -- ($(B1)+(0.3,0.4)$);
\draw[thick] (B1) -- ($(B1)+(-0.6,0)$);

% Bottom right polytope
\coordinate (B2) at (-0.9,-1.3);
\draw[poly] ($(B2)+(-0.3,-0.4)$) % left
    -- ($(B2)+(0.4,-0.3)$)
    -- ($(B2)+(0.5,0.2)$) % right
    -- ($(B2)+(-0.3,-0.4)$)
    -- ($(B2)+(0.5,0.2)$)
    -- ($(B2)+(-0.2,0.4)$) 
    -- cycle;
\coordinate (B3) at ($(B2)+(-0.3,-0.4)$); % left
\coordinate (B4) at ($(B2)+(0.5,0.2)$); % right

% Dashed convex hull connections
\draw[dashedhull] (A3) -- (B3);
\draw[dashedhull] (A4) -- (B4);
%\draw[dashedhull] (A1)+(0.5,0.2)
\draw[dashedhull] (A5) -- (B5);
\draw[dashedhull] (A6) -- (B6);

% =====================================================
% MAP PHI
% =====================================================

\draw[->, thick] (0.7,0) -- (2.0,0)
    node[midway, above] {$\phi$};

% =====================================================
% RIGHT: image in the plane
% =====================================================

% Four points
\coordinate (P1) at (4.3,0.6);
\coordinate (P2) at (3.9,-0.4);
\coordinate (P3) at (5.7,0.6);
\coordinate (P4) at (5.7,-0.3);

\fill[point] (P1) circle (1.5pt);
\fill[point] (P2) circle (1.5pt);
\fill[point] (P3) circle (1.5pt);
\fill[point] (P4) circle (1.5pt);

% Connecting opposite pairs
\draw[thick] (P1) -- (P4);
\draw[thick] (P2) -- (P3);

% The plane
\coordinate (Q1) at (2.5,-1.0);
\draw[poly] (Q1)
    -- ($(Q1)+(3.1,0)$)
    -- ($(Q1)+(3.1,0)+(1.6,2.0)$)
    -- ($(Q1)+(1.6,2.0)$)
    -- cycle;

\end{tikzpicture}

\caption{Illustration of the condition in Theorem~\ref{thm:holmsen}, assuming the four convex sets on the left make an independent set in $M$: intersecting convex hulls of the image points under $\phi$ (on the right) imply intersecting convex hulls of the convex sets (on the left). The picture draws inspiration from~\cite{sadovek2026thesis}.}
    \label{fig:holmsen}
\end{figure}

Given a coloring $\F = \CC_1 \cup \dots \cup \CC_n$ into $n$ colors, we will say that a subfamily $\I \subseteq \F$ is \emph{rainbow} if $|\I \cap \CC_i| \le 1$, for each $i=1, \dots, n$.
Assuming all color classes are nonempty, they define a rank $n$ \emph{partition matroid} on the ground set $\F$ with rainbow sets as independent. Kalai and Meshulam obtained far-reaching topological generalizations of Theorem~\ref{thm:col-helly} to matroids that are not necessarily partition matroids, one of which is the following:

\begin{theorem}[{Kalai and Meshulam~\cite{kalai2005topological}}] \label{thm:kalai-meshulam}
    Let $\F$ be a finite family of convex sets in $\R^d$ and $M$ a matroid on the ground set $\F$. Suppose that $\bigcap \I \neq \emptyset$, for every independent set $\I$ in $M$. Then, there exists a subset $\G \subseteq \F$ such that $\bigcap \G \neq \emptyset$ and $\rank(\F \setminus \G) \le d$. 
\end{theorem}

Moving to the opposite extreme, from points ($k = 0$) to hyperplanes ($k = d-1$), Goodman, Pollack, and Wenger in a series of three papers \cites{goodman1988hadwiger,wenger1990generalization, pollack1990necessary} established a necessary and sufficient condition for a family of convex sets in $\R^d$ to admit a hyperplane transversal. (While we don't state the condition in this paragraph, it can be read off from Theorem~\ref{thm:holmsen} below by substituting $M$ to be the partition matroid of a coloring $\F = \CC_1 \cup \dots \cup \CC_{r+2}$, where $\CC_1 = \dots = \CC_{r+2}$.) This celebrated result in geometric transversal theory, which extended previous work by Hadwiger~\cite{hadwiger1957eibereiche} on line transversals in the plane, has been expanded upon in a variety of ways \cites{Anderson1996Oriented,ArochaSpearoids2002}.
In particular, Arocha, Bracho, and Montejano~\cite{arocha2008colorful} obtained a colorful generalization of Hadwiger's result and conjectured a colorful version of the Goodman-Pollack-Wenger theorem for general $d$ and $k=d-1$. Some partial results were obtained by Holmsen and Rold\'an-Pensado~\cite{holmsen2016colored} and later Holmsen proved a more general matroidal extension of the Goodman-Pollack-Wenger theorem:

\begin{theorem}[{Holmsen \cites{holmsen2022colorful, cheong2024some}}] \label{thm:holmsen}
    Let $\F$ be a finite family of convex sets in $\R^d$ and $M$ a matroid on the ground set $\F$. Suppose that for some $0 \le r < d$ there exists a function $\phi \colon \F \to \R^r$ such that
    \begin{equation*}
        \conv \left(\bigcup \phi(\G_1)\right) \cap \conv \left( \bigcup \phi(\G_2)\right) \neq \emptyset ~~ \Longrightarrow ~~ \conv \left(\bigcup \G_1 \right) \cap \conv \left(\bigcup \G_2 \right) \neq \emptyset,
    \end{equation*}
    whenever $\G_1 \cup \G_2$ is independent in $M$. Then, there exists a subset $\G \subseteq \F$ that has a hyperplane transversal and satisfies $\rank(\F \setminus \G) \le r+1$. 
\end{theorem}

The condition in the theorem above is illustrated in Figure~\ref{fig:holmsen}.
A colorful version is obtained by letting $M$ be the partition matroid of a coloring $\F = \CC_1 \cup \dots \cup \CC_{r+2}$ and, as already mentioned, the Goodman-Pollack-Wenger theorem \cites{goodman1988hadwiger,wenger1990generalization, pollack1990necessary} corresponds to the special case $\CC_1 = \dots = \CC_{r+2}$.

In their first of three papers~\cite{goodman1988hadwiger}, Goodman and Pollack posed the problem of characterizing when a finite family of convex sets in $\mathbb{R}^d$ admits a $k$-transversal, for any $0 \le k \le d-1$.
A first substantial step was obtained by McGinnis \cites{mcginnis2023complex, mcginnis2023necessary}, who established a \emph{complex} analogue of the Goodman--Pollack--Wenger theorem. In this setting, one considers a finite family $\mathcal{F}$ of convex sets in $\mathbb{C}^d$, and the transversal is a complex hyperplane, that is, a complex $(d-1)$-dimensional affine subspace (which has real codimension two).

More recently, McGinnis and the author~\cite{mcginnis2026necessary} resolved the Goodman--Pollack problem in full generality, for all $0 \le k \le d-1$, in both real and complex settings. The two extreme cases $k=0$ and $k=d-1$ recover Helly's and the Goodman-Pollack-Wenger theorem, respectively. However, the characterization of the existence of a $k$-transversal (over a field $\mathbb{F} \in \{\mathbb{R}, \mathbb{C}\}$) is inherently linear-algebraic.
Roughly speaking, a family $\mathcal{F}$ in $\FF^d$ satisfies the condition if there exists a function $\phi : \mathcal{F} \to \mathbb{F}^k$ such that any $d-k$ affine $\FF$-dependencies of points in $\im(\phi)$ can be pulled back in a \emph{proportional} way to affine $\FF$-dependencies of points in elements of $\F$. The precise formulation of the condition may be found in Definition~\ref{def:model-col-dep} below by setting all color classes $\mathcal{F}_i$ to be equal and we refer the reader to Figure~\ref{fig:pullback} for illustration.

Despite this progress, it has been of recurring interest in the community whether a more general \emph{colorful} solution to the Goodman--Pollack problem can be obtained, one that would unify and extend the colorful results described above:

\begin{problem}[{Colorful Goodman-Pollack problem, \cite[Sec.~6]{cheong2024some}, \cite[Sec.~5]{mcginnis2026necessary}}] \label{prob:col-G-P}
   Let $0 \le k < d$ be integers and $\F$ a finite family of convex sets $\F$ in $\R^{d}$ together with a coloring $\F = \CC_1 \cup \dots \cup \CC_n$ into $n = (d-k)(k+1)+1$ colors. Find a \emph{condition} for which the following is true: if $\F$ is such that any rainbow choice of convex sets $F_1 \in \CC_1, \dots, F_n \in \CC_n$ satisfies \emph{the condition}, then there exists a color $i$ such that $\CC_i$ admits a $k$-transversal.
\end{problem}

\subsection*{Our main result}

We provide a solution to the colorful Goodman-Pollack problem (see Corollary~\ref{cor:colorful}) by establishing a more general matroidal Theorem~\ref{thm:main}, extending all of the previous theorems. 
We introduce the following notion (see also~\cite[Def.~1.2]{mcginnis2026necessary}):

\begin{definition} \label{def:model-M,k-dep}
    Let $0 \le r \le k < d$ be integers and $\FF \in \{\R, \C\}$ a field. For a family of convex sets $\F$ in $\FF^d$, and a matroid $M$ on the ground set $\F$. A function
    \begin{equation*}
        \phi \colon \F \longrightarrow \FF^r
    \end{equation*}
    is said to \emph{model $(d-k, M,\FF)$-dependencies of $\F$} if the following condition is satisfied: for any independent set $\I$ in $M$ and any $d-k$ affine $\FF$-dependencies
    \begin{equation*}
    \sum_{F \in \I} a_{j,F} = 0 \in \FF,~ \sum_{F \in \I} a_{j,F} \phi(F) = 0 \in \FF^r, \hspace{3mm} \text{for}~ j=1, \dots, d-k,
\end{equation*}
which are not all trivial (i.e., some $a_{j,F} \neq 0$), there exists real numbers $r_F \ge 0$ and points $q_F \in F$, for every $F \in \I$, such that $d-k$ affine $\FF$-dependencies
\begin{equation*}
    \sum_{F \in \I} r_F a_{j,F} = 0 \in \FF,~ \sum_{F\in \I} r_F a_{j,F} q_F = 0 \in \FF^d, \hspace{3mm} \text{for}~ j=1, \dots, d-k,
\end{equation*}
are not all trivial (i.e., some $r_F a_{j,F} \neq 0$).
\end{definition}

\begin{figure}
    \centering
    \begin{tikzpicture}[
    scale=0.8,
    poly/.style={draw=black, thick},
    plane/.style={draw=black},
    dashedhull/.style={dashed},
    point/.style={circle, fill=black, inner sep=1.5pt},
    orpoint/.style={circle, fill=YellowOrange, inner sep=1.5pt},
    >=Latex
]

% =====================================================
% LEFT: four polytopes
% =====================================================

% Top left polytope
\coordinate (A1) at (-4,0.9);
\draw[poly] ($(A1)+(-0.4,-0.3)$) %left
    -- ($(A1)+(0.3,-0.4)$)
    -- ($(A1)+(0.5,0.2)$) %right
    -- ($(A1)+(0.0,0.5)$)
    -- ($(A1)+(0.3,-0.4)$)
    -- ($(A1)+(-0.4,-0.3)$)
    -- ($(A1)+(0.0,0.5)$);
\coordinate (A3) at ($(A1)+(-0.4,-0.3)$); % left
\coordinate (A4) at ($(A1)+(0.5,0.2)$); % right

% Bottom left polytope
\coordinate (A2) at (-4.5,-1.2);
\draw[poly] ($(A2)+(-0.5,-0.2)$)
    -- ($(A2)+(0.2,-0.5)$) % bottom
    -- ($(A2)+(-0.1,0.4)$)
    -- ($(A2)+(0.2,-0.5)$)
    -- ($(A2)+(0.5,0.1)$)
    -- ($(A2)+(-0.1,0.4)$) %top
    -- cycle;
\coordinate (A5) at ($(A2)+(-0.1,0.4)$); % top
\coordinate (A6) at ($(A2)+(0.2,-0.5)$) ; % bottom

% Top right polytope
%\coordinate (B1) at (-1.0,1.0);
%\draw[poly] ($(B1)+(-0.5,-0.2)$)
%    -- ($(B1)+(0.2,-0.6)$) %bottom
%    -- ($(B1)+(0.6,0.0)$)
%    -- ($(B1)+(0.0,0.5)$) %
%    -- cycle;

% Top right polytope
\coordinate (B1) at (-0.9,1.0);
\draw[poly] ($(B1)+(0,-0.5)$) %bottom
    -- ($(B1)+(0.3,-0.1)$)
    -- ($(B1)+(0.3,0.4)$)
    -- ($(B1)+(-0.3,0.4)$) %top
    -- ($(B1)+(-0.6,0)$)
    -- ($(B1)+(-0.6,-0.5)$)
    -- cycle;
\coordinate (B5) at ($(B1)+(-0.3,0.4)$); % top
\coordinate (B6) at ($(B1)+(0,-0.5)$) ; % bottom
\draw[thick] (B1) -- (B6);
\draw[thick] (B1) -- ($(B1)+(0.3,0.4)$);
\draw[thick] (B1) -- ($(B1)+(-0.6,0)$);

% Bottom right polytope
\coordinate (B2) at (-0.9,-1.3);
\draw[poly] ($(B2)+(-0.3,-0.4)$) % left
    -- ($(B2)+(0.4,-0.3)$)
    -- ($(B2)+(0.5,0.2)$) % right
    -- ($(B2)+(-0.3,-0.4)$)
    -- ($(B2)+(0.5,0.2)$)
    -- ($(B2)+(-0.2,0.4)$) 
    -- cycle;
\coordinate (B3) at ($(B2)+(-0.4,-0.4)$); % left
\coordinate (B4) at ($(B2)+(0.5,0.2)$); % right

% Dashed convex hull connections
\coordinate (AUL) at ($(A1)+(0.3,0.0)$);
\coordinate (BBR) at ($(B2)+(-0.1,0.1)$);
\coordinate (ABL) at ($(A2)+(0.22,0.05)$);
\coordinate (BUR) at ($(B1)+(-0.4,-0.2)$);
\draw[YellowOrange, thick] (AUL) -- (BBR);
\draw[YellowOrange, thick] (ABL) -- (BUR);
\fill[orpoint] (AUL) circle (1.5pt);
\fill[orpoint] (BBR) circle (1.5pt);
\fill[orpoint] (ABL) circle (1.5pt);
\fill[orpoint] (BUR) circle (1.5pt);

% =====================================================
% MAP PHI
% =====================================================

\draw[->, thick] (0.7,0) -- (2.0,0)
    node[midway, above] {$\phi$};

% =====================================================
% RIGHT: image in the plane
% =====================================================

% Four points
\coordinate (P1) at (4.3,0.6);
\coordinate (P2) at (3.9,-0.4);
\coordinate (P3) at (5.7,0.6);
\coordinate (P4) at (5.7,-0.3);

\fill[point] (P1) circle (1.5pt);
\fill[point] (P2) circle (1.5pt);
\fill[point] (P3) circle (1.5pt);
\fill[point] (P4) circle (1.5pt);

% Connecting opposite pairs
\draw[thick] (P1) -- (P4);
\draw[thick] (P2) -- (P3);

% The plane
\coordinate (Q1) at (2.5,-1.0);
\draw[poly] (Q1)
    -- ($(Q1)+(3.1,0)$)
    -- ($(Q1)+(3.1,0)+(1.6,2.0)$)
    -- ($(Q1)+(1.6,2.0)$)
    -- cycle;

\end{tikzpicture}
    \caption{An illustration of the condition from~\cite{mcginnis2026necessary} in the case of a hyperplane transversal ($k=d-1$): an affine dependency of points in the image of $\phi$ is pulled back to an affine dependency of points in elements of $\F$. The picture draws inspiration from~\cite{sadovek2026thesis}.}
    \label{fig:pullback}
\end{figure}

The $k=d-1$ case of Definition~\ref{def:model-M,k-dep} is illustrated in Figure~\ref{fig:pullback}, assuming the four convex sets on the left hand side of the picture form an independent set in the matroid.

In what follows, we will say that an $\FF$-affine $k$-dimensional subspace $V \subseteq \FF^d$ is an \emph{$\FF$-$k$-transversal} to a family of convex sets $\F$ in $\FF^d$ if $V$ intersects every element of $\F$. Relying on a method by Holmsen~\cite{holmsen2022colorful} we will prove the following, which is the main result of the paper:

\begin{theorem} \label{thm:main}
    Let $0 \le r \le k < d$ be integers, $\FF \in \{\R, \C\}$ a field, $\F$ a finite family of compact convex sets in $\FF^d$, and $M$ a matroid on the ground set $\F$. Suppose that there exists a function $\phi \colon \F \to \FF^r$ that models $(d-k, M,\FF)$-dependencies of $\F$. Then, there exists a subfamily $\G \subseteq \F$ which has an $\FF$-$k$-transversal for which
    \begin{equation*}
        \rank(\F \setminus \G) \le \dim_\R(\FF) \cdot (d-k)(r+1).
    \end{equation*}
\end{theorem}

The theorem is trivially true if $\rank(M) \le \dim_\R(\FF) \cdot (d-k)(r+1)$ by taking $\G = \emptyset$. As noted in \cite[Rem.~3.2]{mcginnis2026necessary}, the $k=0$ case of the above theorem recovers Theorem~\ref{thm:kalai-meshulam} of Kalai and Meshulam, while the $k=d-1$ case yields Holmsen's Theorem~\ref{thm:holmsen}. Finally, Theorem~\ref{thm:main} represents a matroidal solution to the colorful Goodman-Pollack problem (see Problem~\ref{prob:col-G-P}), extending the earlier noncolorful result by McGinnis and the author~\cite{mcginnis2026necessary}.

We highlight the special case when $M$ corresponds to the \emph{partition matroid} of a coloring into nonempty color classes. In this specific setting Definition~\ref{def:model-M,k-dep} then takes the following form:

\begin{definition} \label{def:model-col-dep}
    Let $0 \le r \le k < d$ be integers, $\FF \in \{\R, \C\}$ a field, and
    \[
        \F = \CC_1 \cup \dots \cup \CC_n
    \]
    a partition of a family of convex sets in $\FF^d$ into $n$ nonempty color classes. A function $\phi \colon \F \to \FF^r$  is said to \emph{model $d-k$ colorful $\FF$-dependencies of $\F$} if the following condition is satisfied: for any choice $F_1 \in \CC_1, \dots, F_n \in \CC_n$ and any $d-k$ affine $\FF$-dependencies
    \begin{equation*}
        \sum_{\ell = 1}^n a_{j,\ell} = 0 \in \FF,~ \sum_{\ell = 1}^n a_{j,\ell} \phi(F_\ell) = 0 \in \FF^r, \hspace{3mm} \text{for}~ j=1, \dots, d-k,
    \end{equation*}
    which are not all trivial (i.e., some $a_{j,\ell} \neq 0$), there exists real numbers $r_1, \dots, r_n \ge 0$ and points $q_1 \in F_1, \dots, q_n \in F_n$ such that $d-k$ affine $\FF$-dependencies
    \begin{equation*}
        \sum_{\ell = 1}^n r_\ell a_{j,\ell} = 0 \in \FF,~ \sum_{\ell = 1}^n r_\ell a_{j,\ell} q_\ell = 0 \in \FF^d, \hspace{3mm} \text{for}~ j=1, \dots, d-k,
    \end{equation*}
    are not all trivial (i.e., some $r_\ell a_{j,\ell} \neq 0$). 
\end{definition}

We now state the colorful special case of Theorem~\ref{thm:main}:

\begin{corollary} \label{cor:colorful}
    Let $0 \le r \le k < d$ be integers, $\FF \in \{\R, \C\}$ a field, and 
    \[
        \F = \CC_1 \cup \dots \cup \CC_n
    \]
    a partition of a finite family of convex sets in $\FF^d$ into
    \[
        n \ge \dim_\R(\FF) \cdot (d-k)(r+1) + 1
    \]
    nonempty color classes. Suppose that there exists a function $\phi \colon \F \to \FF^r$ that models $d-k$ colorful $\FF$-dependencies of $\F$. Then, there exists a color $i$ such that $\CC_i$ has an $\FF$-$k$-transversal.
\end{corollary}

As noted above, the $k=0$ and $k=d-1$ cases of the corollary correspond, respectively, to Theorem~\ref{thm:col-helly} of Lov\'asz, and the colorful case of Holmsen's Theorem~\ref{thm:holmsen}.

\subsection*{Matroidal Dol'nikov-type result}

As an application of our main result, we obtain Theorem~\ref{thm:dolnikov-matroids} below that can be thought of as a matroidal Dol'nikov-style theorem. First, we recall Dol'nikov's celebrated result on the existence of $k$-transversals. It is a generalization of the central transversal theorem, which was independently discovered by \v{Z}ivaljevi\'c and Vre\'cica~\cite{zivaljevic1990extension}:

\begin{theorem}[{Dol'nikov \cites{dol1993transversals, Dolnikov:1992ut}}] \label{thm:dolnikov}
    Let $0 \le k < d$ be integers and suppose that
    $\F_1,\dots,\F_{k+1}$ are finite families of convex sets in $\mathbb{R}^d$ such that for each $j=1, \dots, k+1$, every $d-k+1$ or fewer sets in $\F_j$ have a nonempty intersection. Then the family $\F_1 \cup \dots \cup \F_{k+1}$ has a (real) $k$-transversal.
\end{theorem}

A complex analogue of the result above was recently obtained in~\cite{mcginnis2026necessary}.

By Helly's classical theorem~\cite{helly1923mengen} intersection of any $d+1$ convex sets in a family $\F$ ensures that the entire family intersects. 
Dol'nikov theorem can be seen a transversal analogue of Helly's theorem in which it is required that only every $d-k+1$ sets in $\F$ intersect in order to conclude the existence of a $k$-transversal; in fact, the intersection condition is even weaker, as it suffices to require the intersection property merely for those $(d-k+1)$-tuples contained in one of $k+1$ prescribed subfamilies of $\F$.
Transversal analogues of colorful Helly theorem have been extensively studied over the past decades.
In this setting, the family $\F$ is partitioned into color classes, and one assumes that every rainbow choice of sets has nonempty intersection in order to conclude that at least one color class admits a $k$-transversal. Results of this kind with sizes of color classes bounded from above were obtained by Montejano and Karasev~\cite{montejano2013transversals} and Holmsen~\cites{holmsen2021transversal, cheong2024some}. See also the work of Karasev~\cite{karasev2009theorems}, Montejano~\cite{montejano2013transversals} and Strausz~\cite{strausz2022how}.

As a consequence of our main theorem, we derive a matroidal generalization of Dol'nikov’s theorem (see Theorem~\ref{thm:dolnikov-matroids} below) together with a colorful variant (see Corollary~\ref{thm:dolnikov-colorful}). 

\begin{theorem} \label{thm:dolnikov-matroids}
    Let $0 \le r \le k < d$ be integers and $\FF \in \{\R, \C\}$ a field.
    Let $\F = \F_1 \cup \dots \cup \F_{r+1}$ be a family of convex sets in $\FF^d$ and $M$ a matroid on the ground set $\F$.
    Suppose that $\bigcap \I \neq \emptyset$ holds for every independent set $\I \subseteq \F_j$ in $M$ of rank
    \begin{equation*}
        \rank(\I) \le \dim_\R(\FF)\cdot (d-k) + 1,
    \end{equation*}
    where $j=1, \dots, r+1$. Then, there exists $\G \subseteq \F$ such that $\G$ has an $\FF$-$k$-transversal and
    \begin{equation*}
         \rank(\F \setminus \G) \le \dim_\R(\FF) \cdot (d-k)(r+1).
    \end{equation*}
\end{theorem}

Theorem~\ref{thm:kalai-meshulam} of Kalai and Meshulam is the $k=r=0$ and $\FF=\R$ case of the theorem above, while Dol'nikov's Theorem~\ref{thm:dolnikov} is the noncolorful version of the colorful Corollary~\ref{thm:dolnikov-colorful} below.

\begin{corollary}\label{thm:dolnikov-colorful}
    Let $0 \le r \le k < d$ be integers, $\FF \in \{\R, \C\}$ a field, and $\F = \F_1 \cup \dots \cup \F_{r+1}$ a family of convex sets in $\FF^d$ that is also colored into
    \begin{equation*}
        n \coloneqq \dim_\R(\FF) \cdot (d-k)(r+1)+1
    \end{equation*}
    nonempty color classes $\F = \CC_1 \cup \dots \cup \CC_n$.
    Suppose that for all $j=1, \dots, r+1$ and every rainbow subfamily $\G_j \subseteq \F_j$ of size at most $\dim_\R(\FF)\cdot (d-k)+1$, we have
    \begin{equation*}
        \bigcap \G_j \neq \emptyset.
    \end{equation*}
    Then, there exists a color $i $ such that $\CC_i$ has an $\FF$-$k$-transversal.
\end{corollary}

We also mention extensions of colorful Helly theorem obtained by Arocha, B{\'a}r{\'a}ny, Bracho, Fabila, and Montejano~\cite{arocha2009very} and Mart\'inez-Sandoval, Rold\'an-Pensado, and Rubin~\cite{martinez2020further}. Additionally, colorful $(p,q)$- and fractional Helly-type theorems were obtained by B\'ar\'any, Fodor, Montejano, Oliveros, and P\'or~\cite{barany2014colourful}, while a fractional, colorful, and spherical $k$-transversal theorems for fat convex sets or balls were recently obtained by Jung and P{\'a}lv{\"o}lgyi~\cites{jung2025k, jung2024note}.

\subsection*{Outline of the paper} In Section~\ref{sec:matroidal-joins} we introduce matroidal spaces and joins, and derive some of their properties. In Section~\ref{sec:non-existence-equiv} we adapt previous nonexistence results for equivariant maps from Stiefel manifolds to spheres to obtain similar result on the skeleta. In Section~\ref{sec:aux-matr-space} we define a particular matroidal space needed for the proof and record some of its properties, which are then used in Section~\ref{sec:proof-main} to prove the main theorem. Finally, in Section~\ref{sec:dolnikov-matroid} we derive the matroidal Dol'nikov Theorem~\ref{thm:dolnikov-matroids} from the main theorem.

\section{Matroidal joins}
\label{sec:matroidal-joins}

In this section we introduce the notion of \emph{matroidal spaces} (see Definition~\ref{def:matroidal-space}) and more specific \emph{matroidal joins} of spaces (see Definition~\ref{def:matroidal-join}); for the latter we derive connectivity bounds (see Lemma~\ref{lem:conn-matroidal-join}). Before doing that, we recall the notion of the Bousfield-Kan bar construction of the homotopy colimit of a diagram; for details see the book by Bousfield and Kan~\cite{bousfield1972homotopy} or the expository note by Dugger~\cite{dugger2008primer}.
Namely, for a small category $\CC$ and a diagram $d \colon \CC \to \catTop$ of topological spaces, the \emph{homotopy colimit} of $d$ is defined as the quotient
\begin{equation*} \label{eq:hocolim}
    \hocolim_\CC(d) = \left(\bigsqcup_{n \ge 0} \bigsqcup_{c_n \to \dots \to c_0 \in \CC} d(c_n) \times \Delta_n \right)/\sim,
\end{equation*}
where $\Delta_n = \{(t_0, \dots, t_n) \in [0,1]^n \colon t_0 + \dots + t_n = 1\}$ is the $n$-simplex and the second disjoint sum goes over all chains of morphisms in $\CC$ of length $n+1$. Moreover, $\sim$ is a relation on the disjoint union generated by the following two relations:
\begin{itemize}
    \item for each $n \ge 1$ and $0 \le i \le n$ an element
    \begin{equation*}
         (x; t_0, \dots, t_n) \in \{c_n \to \dots \to c_0\} \times d(c_n) \times \Delta_n,
    \end{equation*}
    where $t_i = 0$,
    is in relation with
    \begin{equation*}
        \begin{cases}
            (x; t_0, \dots, \widehat{t_i}, \dots, t_n) \in \{c_n \to \dots \to \widehat{c_i} \to \dots \to c_0\} \times d(c_n) \times \Delta_{n-1} & 0 \le i < n\\
            (d(c_n \to c_{n-1})(x); t_0, \dots, t_{n-1}) \in\{c_{n-1} \to \dots \to c_0\} \times d(c_{n-1}) \times \Delta_{n-1} & i = n,
        \end{cases}
    \end{equation*}
    and
    \item for each $n \ge 1$ and $0 \le i \le n-1$ an element
    \begin{equation*}
         (x; t_0, \dots, t_n) \in \{c_n \to \dots \to c_0\} \times d(c_n) \times \Delta_n,
    \end{equation*}
    where $c_{i+1} = c_i$ and the morphism $c_{i+1} \to c_i$ is the identity,
    is in relation with
    \begin{equation*}
            (x; t_0, \dots, t_i + t_{i+1}, \dots, t_n) \in \{c_n \to \dots c_{i+1} \to c_{i-1} \to \dots \to c_0\} \times d(c_n) \times \Delta_{n-1}.
    \end{equation*}
\end{itemize}
In particular, every point in the homotopy colimit has a unique representative
\begin{equation*}
    (x; t_0, \dots, t_n) \in \{c_n \to \dots \to c_0\} \times d(c_n) \times \Delta_n,
\end{equation*}
where none of the morphisms $c_{i+1} \to c_i$ are identities and none of the $t_i$ are zero. 

For a simplicial complex $K$ we will denote by $P(K)$ the \emph{face poset} of $K$. Our convention is that the empty face is not an element of $P(K)$. We will also consider $P(K)$ as a small category which has a unique morphism $\sigma \to \tau$ whenever $\tau \subseteq \sigma$. In the remainder of the paper we will identify the matroid with its \emph{matroidal complex} (i.e., the simplicial complex of independent sets) and, as it is apparent from our definitions, we will only work with \emph{ordered} matroids (i.e., those with an ordering of the ground set).

\begin{definition} \label{def:matroidal-space}
    Let $n \ge 0$ be an integer and $M$ a matroid on the set of vertices $\{0,1, \dots, n\}$. The \emph{matroidal space} over $M$ is the homotopy colimit of a functor $d \colon P(M) \to \catTop$.
\end{definition}

By the above description, a point in $\hocolim_{P(K)}(d)$ can be uniquely presented as
\begin{equation*}
    \left(\sigma, (t_i \colon i \in \sigma), x_\sigma \right)
\end{equation*}
where $\sigma \in P(M)$ is a face of $M$, the $t_i > 0$ are real numbers with $\sum_{i \in \sigma} t_i = 1$, and $x_{\sigma} \in d(\sigma)$.

Let $G$ be a group acting on the $d(\sigma)$ such that all of the $d(\sigma \to \tau) \colon d(\sigma) \to d(\tau)$ are $G$-maps. This defines a $G$-action on $\hocolim_{P(K)}(d)$ by the rule
\[
    g \cdot (\sigma, (t_i \colon i \in \sigma), x_\sigma) = (\sigma, (t_i \colon i \in \sigma), g \cdot x_\sigma), \hspace{3mm} \text{for}~ g \in G.
\]
We will call such matroidal space \emph{$G$-matroidal}. For a real $G$-representation $R$ let $\{f_i \colon d(i) \to R\}_{0 \le i \le n}$ be a collection of $G$-maps. This defines a $G$-map
\begin{align} \label{eq:equiv-map-hocolim}
    \hocolim_{P(K)}(d) &\longrightarrow R\\
    (\sigma, (t_i \colon i \in \sigma), x_\sigma) & \longmapsto \sum_{i \in \sigma} t_i \cdot (f \circ d(\sigma \to i))(x_{\sigma}). \notag
\end{align}
We introduce the appropriate notion of a subspace:

\begin{definition} \label{def:matroidal-subspace}
    Let $L \subseteq K$ be two matroids, and $d \colon P(K) \to \catTop$, $d' \colon P(L) \to \catTop$ two diagrams. We say that $\hocolim_{P(L)}(d')$ is a \emph{matroidal subspace of $\hocolim_{P(K)}(d)$ over $L$} if there is a natural transformation between $d'$ and $d|_{P(L)}$ which is an objectwise inclusion.
\end{definition}

\begin{figure}
    \centering

\begin{tikzpicture}[
    scale=0.8,
    poly/.style={draw=black, thick},
    plane/.style={draw=black},
    dashedhull/.style={dashed},
    point/.style={circle, fill=black, inner sep=1.5pt},
    >=Latex
]

% The points
\coordinate (M1) at (0.0,0.0);
\coordinate (M12) at (2.4,0.0);
\coordinate (M01) at (-1.7,-1.2);
\coordinate (M2) at ($2*(M12)$);
\coordinate (M0) at ($2*(M01)$);

\fill[point] (M1) circle (1.5pt);
\node at ($(M1)-(0.0,0.4)$) {$1$};
    
\fill[point] (M2) circle (1.5pt);
\node at ($(M2)-(0.0,0.4)$) {$2$};
    
\fill[point] (M0) circle (1.5pt);
\node at ($(M0)-(0.0,0.4)$) {$0$};
    
\fill[point] (M12) circle (1.5pt);
\node at ($(M12)-(0.0,0.4)$) {$12$};
    
\fill[point] (M01) circle (1.5pt);
\node at ($(M01)-(0.0,0.4)$) {$01$};

% Drawing the matroid
\draw[] (M1) -- (M0);
\draw[] (M1) -- (M2);

% Written the X's
\coordinate (V) at (0.0,0.9);

\draw[dashedhull] (M1) -- ($(M1)+(V)$);
\node at ($(M1)+(V)+(0.0,0.4)$) {$X_1$};

\draw[dashedhull] (M0) -- ($(M0)+(V)$);
\node at ($(M0)+(V)+(0.0,0.4)$) {$X_0$};

\draw[dashedhull] (M2) -- ($(M2)+(V)$);
\node at ($(M2)+(V)+(0.0,0.4)$) {$X_2$};

\draw[dashedhull] (M01) -- ($(M01)+(V)$);
\node at ($(M01)+(V)+(0.0,0.4)$) {$X_0 \times X_1$};

\draw[dashedhull] (M12) -- ($(M12)+(V)$);
\node at ($(M12)+(V)+(0.0,0.4)$) {$X_1 \times X_2$};

\end{tikzpicture}

\caption{Illustration of a matroidal join $M(X_0, X_1, X_2)$ over the matroid $M$ with bases $\{0,1\},\{1,2\}$.}
    \label{fig:matroidal-join}
\end{figure}

A key property of homotopy colimits is that they preserve \emph{weak equivalences} (these are maps that induce isomorphisms on homotopy groups): if $f \colon d \Rightarrow d'$ is an objectiwise weak equivalence, then the induced map between homotopy colimits is also a weak equivalence, that is, it induces isomorphisms on homotopy groups; see the book by Bousfield and Kan \cite[Sec.~XII.4.2]{bousfield1972homotopy}. 

\begin{lemma}[\cite{bousfield1972homotopy}] \label{lem:weak-equiv}
    Let $n \ge 0$ be an integer and $M$ a matroid on the set of vertices $\{0,1, \dots, n\}$ and let $d,d' \colon P(K) \to \catTop$ be two diagrams. If there is a natural transformation $f \colon d \Rightarrow d'$ such that $f(\sigma)$ is a weak equivalence, for each $\sigma \in P(K)$, then
    \begin{equation*}
        \hocolim_{P(K)}(f) \colon \hocolim_{P(K)} (d) \longrightarrow \hocolim_{P(K)} (d')
    \end{equation*}
    is a weak equivalence. If additionally functors $d$ and $d'$ take values in CW complexes, then the map $\hocolim_{P(K)}(f)$ is a homotopy equivalence.
\end{lemma}

Here the stronger conclusion for CW complexes holds by the Whitehead's theorem~\cite{whitehead1949combinatorial}. While we will use the full notion of matroid spaces in later sections, in the spacial case where the functor $d$ depends only on the values on vertices of $M$ we have the following notion (see also Figure~\ref{fig:matroidal-join}):

\begin{definition} \label{def:matroidal-join}
    Let $n \ge 0$ be an integer and $M$ a matroid on the set of vertices $\{0,1, \dots, n\}$. The \emph{matroidal join} over $M$ of nonempty topological spaces $X_0, \dots, X_n$ is defined as the matroidal space
    \begin{equation*}
        M(X_0, \dots, X_n) \coloneqq \hocolim_{P(M)}(d_{X_0, \dots, X_n}),
    \end{equation*}
    where $d_{X_0, \dots, X_n} \colon P(M) \to \catTop$ is the product diagram
    \begin{equation*}
        d_{X_0, \dots, X_n}(\sigma) = \prod_{i \in \sigma} X_i, \hspace{3mm} \text{for}~\sigma \in P(M),
    \end{equation*}
    and
    \begin{equation*}
        d_{X_0, \dots, X_n}(\sigma \to \tau) \colon \prod_{i \in \sigma} X_i \longrightarrow \prod_{j \in \tau} X_j, \hspace{3mm} \text{for}~\sigma \to \tau \in P(M)
    \end{equation*}
    is the projection.
\end{definition}

\begin{example}\hfill

    \begin{enumerate}
        \item If all the $X_i$ are singletons, then $M(X_0, \dots, X_n)$ is equal to (the geometric realization of) $M$.
        \item If, for some $i$, the space $X_i$ is the discrete space of two points, while all the other $X_j$ are singletons, then $M(X_0, \dots, X_n)$ is the \emph{parallel extension} of a matroid $M$ at the vertex $i$ and is itself a matroid of the same rank as $M$; for more details, consult the book by Oxley \cite[Sec.~5.4]{oxley2006matroid}. More generally, if all the $X_i$ are discrete finite spaces, then $M(X_0,\dots,X_n)$ is an iterated parallel extension of $M$, hence is a matroid as well of the same rank as $M$.
        \item If $M$ is the full $n$-simplex $\Delta_n$, then $\Delta_n(X_0, \dots, X_n)$ is the join $X_0 * \dots * X_n$, whose points are formal convex combinations $t_0x_0 + \dots + t_n x_n$, for $(t_0, \dots, t_n) \in \Delta_n$ and $x_i \in X_i$.
    \end{enumerate}
\end{example}

In general, for a matroid $M$, the matroidal join $M(X_0, \dots, X_n)$ can alternatively be expressed as the subspace of the join: 
\begin{equation*}
    M(X_0, \dots, X_n) = \{t_0x_0 + \dots + t_n x_n \in X_0 * \dots * X_n \colon~ \{i \colon t_i > 0\} \in M\}.
\end{equation*}
The main reason why we opted for the (equivalent) definition via homotopy colimits is to quickly exhibit the following connectivity property of $M(X_0, \dots, X_n)$.
The next lemma mimics the inductive proof of Folkman's result~\cite{folkman1966homology} on the connectivity of matroids in terms of their rank; see also strengthenings by Bj\"orner \cites{bjorner1980shellable, bjorner1992}.

\begin{lemma} \label{lem:conn-matroidal-join}
    Let $n \ge 0$ be an integer and $M$ a nonempty matroid on the vertex set $\{0, \dots, n\}$. For nonempty spaces $X_0, \dots, X_n$ the matroidal join $M(X_0, \dots, X_n)$ is $(\rank(M)-2)$-connected. 
\end{lemma}
\begin{proof}
    Without loss of generality we may assume that $M$ is loopless, i.e., $0, \dots, n \in M$.
    For each $i$ there is a weak homotopy equivalence $f_i \colon Y_i \xrightarrow{\sim} X_i$ where $Y_i$ is a CW complex called the \emph{CW approximation of $X_i$}; for details consult the book by Hatcher \cite[Pg.~352-353]{hatcher2002algebraic}. Maps $f_i$ induce a natural transformation $f \colon d_{Y_0, \dots, Y_n} \Rightarrow d_{X_0, \dots, X_n}$ between the two diagrams over $P(M)$ used in the construction of spaces $M(X_0, \dots, X_n)$ and $M(Y_0, \dots, Y_n)$ (see Definition~\ref{def:matroidal-join}).
    By Lemma~\ref{lem:weak-equiv} the two matroidal joins are weakly equivalent,
    so it is enough to prove the lemma assuming $X_i$ are CW complexes.
    
    The rest of the proof is carried out by induction on $n+k \ge 1$, where $k = \rank(M)$. For the (extended) base case $k = 1$ the matroid $M$ is zero-dimensional, so $M(X_0, \dots, X_n)$ is a disjoint union of the nonempty spaces $X_i$, i.e., it is $(-1)$-connected.

    For the inductive step, suppose that $n \ge 1$ and $k \ge 2$.
    If $k = n+1$, that is, if $M$ is the full $n$-simplex, then $M(X_0, \dots, X_n) = X_0 * \dots * X_n$ is the join of nonempty spaces and hence $(n-1)$-connected, which follows by the work of Whitehead~\cite{whitehead1956homotopy} on connectivity of the joins. 
    On the other hand, if $k \le n$, then there exists a vertex of $M$ which is not in every maximal face (basis). Without loss of generality, let that vertex be $n$. Then, the deletion matroid $M\setminus n$ consisting of faces of $M$ not having $n$ as a vertex is also a matroid of rank $k$, so by the induction hypothesis the matroidal join $(M\setminus n)(X_0, \dots, X_{n-1})$ is $(k-2)$-connected. Let $\link_M(n) = \{\sigma \in M\setminus n \colon \sigma \cup \{n\} \in M\}$ denote the link in $M$ of the vertex $n$, which is itself a matroid of rank $k-1$. Let $M_n(X_0, \dots, X_n)$ denote the homotopy colimit of the restriction of the diagram $d$ from Definition~\ref{def:matroidal-join} to the subcategory $P(\link_M(n)) \subseteq P(M)$. By construction, this space is the $\link_M(n)$-matroidal join of a tuple of spaces $(X_i \colon i \in \link_M(n))$, so by the induction hypothesis is $(k-3)$-connected. Next, we observe that there is a pushout diagram
    \begin{equation*}
        \begin{tikzcd}
            M_n(X_0, \dots, X_n) \arrow[r] \arrow[d] & (M \setminus n)(X_0, \dots, X_{n-1}) \arrow[d] \\
            M_n(X_0, \dots, X_n) * X_n \arrow[r] & M(X_0, \dots, X_n).
        \end{tikzcd}
    \end{equation*}
    Moreover, since all $X_i$ are CW complexes, the two arrows in the diagram from the top-left spot are cofibrations.
    Therefore, by the Blakers-Massey theorem~\cite{blakers1951homotopy} it follows that the horizontal maps induce isomorphisms in homotopy groups of pairs
    \begin{align*}
        \pi_\ell&(M_n(X_0, \dots, X_n) * X_n, M_n(X_0, \dots, X_n)) \xrightarrow{\cong} \pi_\ell(M(X_0, \dots, X_n) , (M \setminus n)(X_0, \dots, X_{n-1})),
    \end{align*}
    for $0 \le \ell \le k-2$.
    The claim of the lemma now follows from the long exact sequence of a pair in homotopy groups and the connectivity of the spaces.
\end{proof}

\section{Nonexistence of equivariant maps}
\label{sec:non-existence-equiv}

In this section we state a theorem on nonexistence of equivariant maps from the skeleta of Stiefel manifolds into spheres (see Theorem~\ref{thm:equiv-map}), which essentially follows from previous works of Fadell and Husseini~\cite{fadell1988ideal}, Chan, Chen, Frick, and Hull~\cite{chan2020borsuk}, and the author and Sober\'on~\cite{sadovek2026complex}.

Before stating the result we fix notation.
Let $\FF$ be a field $\R$ or $\C$ and denote by 
\[
    \stiefel{n}{d} = \{(v_1, \dots, v_n) \in S(\FF^d)^n \colon \langle v_i, v_j \rangle_\FF = 0 ~ \text{for all}~i\neq j\}
\]
the Stiefel manifold of orthonormal $n$-frames in $\FF^d$ with respect to the $\FF$-inner product, where $S(\FF^d)$ is the unit sphere in $\FF^d$. We will use multiplicative notation for $\Z_2 = \{-1,1\}$ and we endow $\stiefel{n}{d}$ with a $S(\FF)^n$-action
\begin{equation*}
    (g_1, \dots, g_n) \cdot (v_1, \dots, v_n) = (g_1 v_1, \dots, g_n v_n),
\end{equation*}
for $(g, \dots, g_n) \in S(\FF)^n$ and $(v_1, \dots, v_n) \in \stiefel{n}{d}$. We will also consider an action of the subgroup
$\Z_2^n = S(\R)^n \subseteq S(\FF)^n$.

As already mentioned, we will need the following result on the nonexistence of equivariant maps from Stiefel manifolds, which in the real case is due to Chan, Chen, Frick, and Hull \cite[Thm~1.1]{chan2020borsuk} and is implicit in the work of Fadell and Husseini~\cite{fadell1988ideal}. In the complex case it is a result of the author and Sober\'on \cite[Rem.~5.1]{sadovek2026complex}. These results can be put together to claim the nonexistence of $\Z_2^n$-equivariant maps
\begin{equation} \label{eq:equiv-map}
    \stiefel{n}{d} \longrightarrow S(\FF^{d-1} \oplus \dots \oplus \FF^{d-n}),
\end{equation}
where the codomain is the unit sphere in the $\FF$-vector space $\FF^{d-1} \oplus \dots \oplus \FF^{d-n}$ with the product action.
In what follows, $\sk_k(K)$ denotes the $k$-skeleton of a cellular complex $K$ and $K$ is a called a $G$-CW complex, for a given finite group $G$, if $G$ acts by cellular maps on $K$.

\begin{theorem} \label{thm:equiv-map}
    Let $0 \le n  < d$ and $0 \le k_1 \le d-1, \dots, 0 \le k_n \le d-n$ be integers. Suppose that $\stiefel{n}{d}$ is endowed with a structure of a $\Z_2^n$-CW complex. Then, there does not exist a $\Z_2^n$-equivariant map
    \begin{equation*}
        \sk_{(k_1 + \dots + k_n)\cdot \dim_\R(\FF)}(\stiefel{n}{d}) \longrightarrow S(\FF^{k_1} \oplus \dots \oplus \FF^{k_n}),
    \end{equation*}
    where $\Z_2^n$ has a product antipodal action on $\FF^{k_1} \oplus \dots \oplus \FF^{k_n}$.
\end{theorem}

The proofs in \cites{fadell1988ideal, sadovek2026complex} via Fadell-Husseini index computations can be leveraged directly in the case $k_1 = \dots = k_n = d-n$ to obtain Theorem~\ref{thm:equiv-map}, since the index is created already by the $(n(d-n)\dim_\R(\FF))$-skeleton. We provide an elementary reduction of the full Theorem~\ref{thm:equiv-map} to the nonexistence result of a $\Z_2^n$-map \eqref{eq:equiv-map}.
We will need the following folklore lemma.

\begin{lemma} \label{lem:sk-map}
    Let $G$ be a finite group and $K$ a free $G$-CW complex of dimension $n \ge 1$. Suppose that, for a real $G$-representation $V$, there exists a $G$-equivariant map
    \begin{equation*}
        \sk_{n-1}(K) \longrightarrow S(V).
    \end{equation*}
    Then, for any 1-dimensional $G$-representation $\R$ there exists a $G$-equivariant map
    \begin{equation*}
        K \longrightarrow S(V \oplus \R).
    \end{equation*}
\end{lemma}
\begin{proof}
    Let $f \colon \sk_{n-1}(K) \to S(V)$ be a $G$-map.
    Since $K$ is a free $G$-CW complex, by elementary equivariant obstruction theory, it suffices to show that for an arbitrary $n$-cell $e \subseteq K$ the map
    \begin{equation*}
        g_e \colon \partial e \xrightarrow{~f|_{\partial e}~} S(V) \hookrightarrow S(V \oplus \R)
    \end{equation*}
    can be extended to a map $e \to S(V \oplus \R)$. This is indeed the case, as $g_e$ factors through a homotopically trivial map $S(V) \hookrightarrow S(V \oplus \R)$ and is therefore homotopically trivial.
\end{proof}

\begin{proof}[Proof of Theorem~\ref{thm:equiv-map}]
    Suppose that there exists a $\Z_2^n$-map
    \[
        \sk_{(k_1 + \dots + k_n)\cdot \dim_\R(\FF)}(\stiefel{n}{d}) \longrightarrow S(\FF^{k_1} \oplus \dots \oplus \FF^{k_n}).
    \]
    By successively applying Lemma~\ref{lem:sk-map} $( \dim_\R\FF \cdot \sum_{i=1}^n(d-i -k_i))$-many times, we conclude that there exists a $\Z_2^n$-map \eqref{eq:equiv-map}.
    Here we note for $\FF=\C$, any $\Z_2$-representation $\C$ splits as a direct sum of two 1-dimensional real representations, so the lemma may indeed be applied in this case as well.
    As already explained above, equivariant map \eqref{eq:equiv-map} does not exists~\cites{chan2020borsuk, fadell1988ideal, sadovek2026complex}.
\end{proof}

\section{An auxiliary matroidal space}
\label{sec:aux-matr-space}

In this section we define a particular matroidal space $X$ over $M$ and exhibit its properties, which will be crucial for the proof of Theorem~\ref{thm:main} in Section~\ref{sec:proof-main}.

Let $0 \le r \le k < d$ be integers. We set further notation. As before, let $S(\FF) = \{z \in \FF \colon \|z\| = 1\}$ denote the unit sphere and
\begin{equation*}
    C_{d-k}(\FF) \coloneqq \left( S(\FF) \cup \{0\} \right)^{d-k} \setminus \{(0, \dots, 0\}
\end{equation*}
be the set of nonzero vectors $z=(z_1, \dots, z_{d-k}) \in \FF^{d-k}$ such that $\|z_i\|$ is either zero or one, for each $i=1, \dots, d-k$. In the real case, the set $C_{d-k}(\R)$ can be identified with the face poset of the $(d-k)$-dimensional cross polytope $\{(x_1, \dots, x_{d-k}) \in \R^{d-k} \colon \sum_{i=1}^{d-k} |x_i| \le 1\}$. We endow $C_{d-k}(\FF)$ with the inverse-product $S(\FF)^{d-k}$-action
\begin{equation*}
    (g_1, \dots, g_{d-k}) \cdot (z_1, \dots, z_{d-k}) = (g_1^{-1} \cdot z_1, \dots, g_{d-k}^{-1} \cdot z_{d-k}),
\end{equation*}
for $(g_1, \dots, g_{d-k}) \in S(\FF)^{d-k}$ and $(z_1, \dots, z_{d-k}) \in C_{d-k}(\FF)$.

The key object for us is the matroidal space (recall Definition~\ref{def:matroidal-space}) over $M$
\begin{equation} \label{eq:def-X}
    X = X(M, \F) = \hocolim_{P(M)}(d(M, \F)),
\end{equation}
where the diagram
\begin{equation*}
    d = d(M, \F) \colon P(M) \longrightarrow \catTop
\end{equation*}
is defined as follows. On objects $\{i_1 < \dots < i_m \} \in P(M)$ it takes the form 
\begin{equation} \label{eq:d-value}
    d(\{i_\ell\}_{\ell = 1}^m) =  \Big\{(z_\ell)_{\ell = 1}^m \in (C_{d-k}(\FF))^m \colon \{i_\ell\}_{\ell = 1}^m ~ \text{and} ~ (z_\ell)_{\ell = 1}^m ~ \text{satisfy}~ (c_1) \Big\}
\end{equation}
and the morphisms are induced by the projections $(C_{d-k}(\FF))^m \to (C_{d-k}(\FF))^{m-1}$. The condition $(c_1)$ is:
\begin{enumerate}
    \item[{($c_1$)}] any $d-k$ affine $\FF$-dependencies of the form
\begin{equation*}
    \sum_{\ell=1}^m a_\ell z_{j,\ell} = 0 \in \FF,~ \sum_{\ell=1}^m a_\ell z_{j,\ell} q_\ell = 0 \in \FF^d, \hspace{3mm} \text{for}~ j=1, \dots, d-k,
\end{equation*}
where $a_1, \dots, a_m  \ge 0$ are real numbers and $q_1 \in F_{i_1}, \dots ,q_m \in F_{i_m}$ are points, are necessarily trivial, that is, $a_1, \dots, a_m = 0$.
\end{enumerate}
We extend the $S(\FF)^{d-k}$-action on $C_{d-k}(\FF)$ to the diagonal action on the product $(C_{d-k}(\FF))^m$ and notice that $d(\{i_\ell\}_{\ell = 1}^m)$ is an $S(\FF)^{d-k}$-invariant space: if $(z_\ell)_{\ell = 1}^m$ satisfies ($c_1$) then $(g \cdot z_\ell)_{\ell = 1}^m$ satisfies it as well, for any $g=(g_1, \dots, g_{d-k}) \in S(\FF)^{d-k}$, since the multiplicative term $g_j^{-1}$ may be canceled from the $j$'th affine dependency in ($c_1$). This defines an $S(\FF)^{d-k}$-matroidal space structure on $X$.

Let us denote by $\hat\F \coloneqq \F \times \{1\}$ the set $\F$ lifted to the ``height one'' space $\FF^d \times \{1\} \subseteq \FF^{d+1}$. For a point $q \in F$ we will also write $\hat q \coloneqq (q,1)$ for short. We for a vector $z\in \FF^{d+1}$ and $\hat F \in \hat \F$ we will write $\langle z, \hat F \rangle  > 0$ for the fact that
\begin{equation*}
    \re \langle z, \hat q \rangle_\FF = \langle z, \hat q \rangle_\R > 0, \hspace{3mm} \text{for all}~\hat q \in \hat F.
\end{equation*}
For an orthonormal $(d-k)$-frame $v = (v_1, \dots, v_{d-k}) \in \stiefel{d-k}{d+1}$ and a convex set $\hat{F}_i \in \hat \F$ we have that
\begin{equation*}
    \hat F_i \cap (\linspan_\FF\{v_1, \dots, v_{d-k}\})^{\perp_{\FF}} \neq \emptyset
\end{equation*}
if and only if the $\FF$-orthogonal projection $\FF^{d+1} \to \linspan_\FF\{v_1, \dots, v_{d-k}\}$ does not map $\hat F_i$ to the origin. This is further equivalent to the fact that for every $z \in C_{d-k}(\FF)$ we have $\langle z_1 v_1 + \dots + z_{d-k} v_{d-k}, \hat F_i \rangle = 0$. To each $v \in \stiefel{d-k}{d+1}$ we associate a subspace $S(v) \subseteq \{0, \dots, n\}$ of the vertex set of $M$ defined as
\begin{align} \label{eq:S(v)}
    S(v) \coloneqq &\big\{i \in \{0, \dots, n\} \colon~  \hat F_i \cap (\linspan_\FF\{v_1, \dots, v_{d-k}\})^{\perp_{\FF}} = \emptyset \big\} \notag \\
    = & \big\{i \in \{0, \dots, n\} \colon~  (\exists z \in C_{d-k}(\FF))~\langle z_1 v_1 + \dots + z_{d-k} v_{d-k}, \hat F_i \rangle > 0 \big\}. 
\end{align}
The sets $S(v)$ are $S(\FF)^{d-k}$-invariant in the sense that $S(g \cdot v) = S(v)$, for any $g \in S(\FF)^{d-k}$, and the induced subcomplex $M[S(v)] \subseteq M$ on the set of vertices $S(v)$ is also a matroid. Moreover, since the condition defining $S(v)$ is open in variable $v \in \stiefel{d-k}{d+1}$, we observe that $S(v)$ is \emph{locally inclusion-monotone} in the sense that the function
\begin{equation*}
    S \colon V_{d-k}(\FF^{d+1}) \longrightarrow 2^{[n]},~v \longmapsto S(v)
\end{equation*}
has the property that every $v$ has a neighborhood $v \in N_v \subseteq V_{d-k}(\FF^{d+1})$ such that $S(v) \subseteq S(u)$ for all $u \in N_v$.
We define a matroidal join (recall Definition~\ref{def:matroidal-join}) over the matroid $M[S(v)]$ as
\begin{equation} \label{eq:X(v)}
    X(v) \coloneqq M[S(v)] \left( \{z \in C_{d-k}(\FF) \colon \langle z_1 v_1 + \dots + z_{d-k} v_{d-k}, \hat F_i \rangle > 0\} \colon i \in S(v)\right).
\end{equation}
The following is the key observation:

\begin{lemma} \label{lem:X(v)}
    Let $v \in \stiefel{d-k}{d+1}$.
    The matroidal join $X(v)$ is $(\rank(M[S(v)]) - 2)$-connected and is a matroidal subspace of $X$ over $M[S(v)]$ in the sense of Definition~\ref{def:matroidal-subspace}. Furthermore, there exists a neighborhood $N_v \subseteq \stiefel{d-k}{d+1}$ of $v$ such that $X(v)$ is a matroidal subspace of $X(u)$, for every $u \in N_v$.
\end{lemma}
\begin{proof}
    The connectivity claim follows from Lemma~\ref{lem:conn-matroidal-join}. To show that $X(v)$ is a matroidal subspace of $X$ over $M[S(v)]$ we need to prove that for each $\{i_1, \dots, i_m\} \in M[S(v)]$ 
    \begin{equation*}
        \Big\{(z_\ell)_{\ell = 1}^m \in (C_{d-k}(\FF))^m \colon \{i_\ell\}_{\ell = 1}^m ~ \text{and} ~ (z_\ell)_{\ell = 1}^m ~ \text{satisfy}~ (c_2) \Big\} \subseteq d(\{i_1, \dots, i_m\}),
    \end{equation*}
    where the latter is the space from \eqref{eq:d-value} used to define $X$ and the condition $(c_2)$ is:
    \begin{itemize}
        \item[{($c_2$)}] $\langle z_{1,\ell} v_1 + \dots + z_{d-k,\ell}v_{d-k}, \hat F_{i_\ell} \rangle > 0$, for each $\ell = 1, \dots, m$.
    \end{itemize}
    In other words, we need to show that ($c_2$) implies ($c_1$). To this end, let $a_1, \dots, a_\ell \ge 0$ be real numbers and $\hat q_1 \in \hat F_{i_1}, \dots, \hat q_m \in \hat F_{i_m}$ points such that
    \begin{equation*}
        \sum_{\ell=1}^m a_\ell z_{j,\ell} \hat q_\ell = 0 \in \FF^{d+1}, \hspace{3mm} \text{for}~ j=1, \dots, d-k.
    \end{equation*}
    Taking the inner product of the second equality with $\langle v_j, \cdot \rangle_\FF$, summing over $j=1, \dots, d-k$, and taking the real part, we obtain
    \begin{align*}
        0 = \re \left(\sum_{j=1}^{d-k} \sum_{\ell=1}^m a_\ell z_{j,\ell} \langle v_j, \hat q_\ell \rangle_\FF \right) = \sum_{\ell=1}^m a_\ell \sum_{j=1}^{d-k} \re \langle z_{j,\ell} v_j, \hat q_\ell \rangle_\FF.
    \end{align*}
    Assuming ($c_2$), the second sum on right hand side (the sum over $j$) is strictly positive for each $\ell$; therefore, all the $a_\ell$ are zero, proving condition ($c_1$).

    Finally, the existence of the neighborhood $N_v$ follows from the fact that $S(v)$ is locally inclusion-monotone and from openness of the condition ($c_2$).
\end{proof}

\section{Proof of the main theorem}
\label{sec:proof-main}

In this section we prove Theorem~\ref{thm:main} using a method of Holmsen~\cite{holmsen2022colorful}. We assume the setting from the statement of the theorem with a slight notational change in line with previous sections: $\F = \{F_0, \dots, F_n\}$ is a family of closed convex sets and $M$ is a matroid on $\{0,\dots, n\}$. As noted in Section~\ref{sec:intro}, we may assume without loss of generality that $\rank(M) = \dim_\R(\FF) \cdot (d-k)(r+1) + 1$, because for $M$ of smaller rank the theorem holds trivially for $\G = \emptyset$.

To argue by contradiction we assume that the conclusion of the theorem does not hold: for any $S \subseteq \{0, \dots, n\}$, if 
\[
    \rank(\{0, \dots, n\} \setminus S) < \rank(M) = \dim_\R(\FF) \cdot (d-k)(r+1) + 1,
\]
then the family of convex sets $\{F_i \in \F \colon i \in S\}$ does not have a $k$-transversal. 

First, let us define a space
\begin{equation*}
    E \coloneqq \{(v,x) \in \stiefel{d-k}{d+1} \times X \colon x \in X(v)\}
\end{equation*}
Here $X$ is the matroidal space \eqref{eq:def-X}, $X(v)$ it the matroidal join \eqref{eq:X(v)}, and by Lemma~\ref{lem:X(v)} we have that $X(v) \subseteq X$ as a matroidal subspace over $M[S(v)]$. We endow $E$ with the diagonal $S(\FF)^{d-k}$-action and let $p_1 \colon E \to \stiefel{d-k}{d+1}$ and $p_2 \colon E \to X$ denote the two $S(\FF)^{d-k}$-equivariant projections. In fact, for the remainder of the proof, we will consider the action of the subgroup $\Z_2^{d-k} = S(\R)^{d-k} \subseteq S(\FF)^{d-k}$, unless specified otherwise.

From the assumption at the beginning of the section, it follows that the set $S(v)$ (defined in \eqref{eq:S(v)}) and consequently the space $X(v)$ are nonempty, for every $v \in \stiefel{d-k}{d+1}$. That is, the first projection $p_1$ is surjective. Now, since $\stiefel{d-k}{d+1}$ is a closed smooth Riemannian manifold, the neighborhood $N_v$ from Lemma~\ref{lem:X(v)} may be chosen to be an $\varepsilon$-ball around $v$, where the constant $\varepsilon>0$ is uniform for all $v$ and the ball is considered with respect to the geodesic metric in the Stiefel manifold. By the work of Illman~\cite{illman1978smooth} the Stiefel manifold $\stiefel{d-k}{d+1}$ may be endowed with the structure of a $\Z_2^{d-k}$-simplicial complex and, due to compactness, we may consider a sufficiently fine subdivision such that the diameter of every simplex is less than $\varepsilon/2$.

Our first claim is that there exists a $\Z_2^{d-k}$-equivariant section
\begin{equation*}
    s \colon \sk_{\rank(M)-1}( \stiefel{d-k}{d+1}) \longrightarrow E
\end{equation*}
of $p_1$ from the $(\rank(M)-1)$-skeleton. Indeed, let us show that $s$ can be constructed inductively on the skeleta using elementary equivariant obstruction theory; for more details, consult a book by tom Dieck \cite[Ch.~II]{tom1987transgroups}. For the base case, a section $s_0 \colon \sk_{0}( \stiefel{d-k}{d+1}) \to E$ may be chosen since the fibers $p_1^{-1}(v) = X(v)$ are nonempty, as observed earlier. Suppose that $0 \le a < \rank(M)-1$ is an integer and that a $\Z_2^{d-k}$-equivariant section
\[
    s_a \colon \sk_{a}( \stiefel{d-k}{d+1}) \longrightarrow E
\]
exists. Let $e$ be any $(a+1)$-face and let $v \in e \setminus \partial e$. The diameter of each simplex in the triangulation is less than $\varepsilon/2$, hence by Lemma~\ref{lem:X(v)} it follows that the section $s_a|_{\partial e}$ factors as 
\[
    s_a|_{\partial e} \colon \partial e \longrightarrow \partial e \times X(v) \subseteq p_1^{-1}(\partial e),
\]
where the arrow is the identity on the ($\partial e$)-factor.
By Lemma~\ref{lem:X(v)} the space $X(v)$ is $(\rank(M[S(v)])-2)$-connected. Moreover, we also have that $\rank(S(v)) = \rank(M)$, as otherwise the family of convex sets $\{F_i \in \F \colon i \notin S(v)\}$ would have a $k$-transversal, contradicting our assumption from the beginning of the section. Now, since $\partial e$ is homeomorphic to an $a$-dimensional sphere and $a \le \rank(M) - 2$, it follows that $s_a|_{\partial e}$ may be extended to a section
\begin{equation*}
    s_{a+1}|_e \colon e \longrightarrow e \times X(v) \subseteq p_1^{-1}(e)
\end{equation*}
Finally, since $\stiefel{d-k}{d+1}$ is a free $\Z_2^{d-k}$-complex and $e$ was an arbitrary $(a+1)$-cell, we have that there exists a $\Z_2^{d-k}$-equivariant section $s_{a+1} \colon \sk_{a+1}(\stiefel{d-k}{d+1}) \to E$ extending $s_a$, which completes the induction.

Our next claim is that there exists a $\Z_2^{d-k}$-equivariant map
\begin{equation*}
    \Phi \colon X \longrightarrow (\FF^{r+1})^{d-k} \setminus \{0\}.
\end{equation*}
In fact, the map we will construct will be $S(\FF)^{d-k}$-equivariant, where the action on the codomain is given by
\begin{equation*}
    (g_1, \dots, g_{d-k}) \cdot (y_1, \dots, y_{d-k})  = (g^{-1}_1 \cdot y_1, \dots, g^{-1}_{d-k} \cdot y_{d-k}),
\end{equation*}
for $(g_1, \dots, g_{d-k}) \in S(\FF)^{d-k}$ and $(y_1, \dots, y_{d-k}) \in (\FF^{r+1})^{d-k}$.
For each vertex $i =0. \dots, n$ of $M$ we define an $S(\FF)^{d-k}$-map
\begin{align*}
    \Phi_i \colon d(i)= \{z \in C_{d-k}(\FF) \colon  \{i\} ~ \text{and} ~ z ~ \text{satisfy} ~ (c_1) \} &\longrightarrow (\FF^{r+1})^{d-k}\\
    z &\longmapsto \Big( z_{j}, z_{j} \phi(F_i) \Big)_{j=1}^{d-k}.
\end{align*}
We recall that $\phi \colon \F \to \FF^r$ is the function from the statement of Theorem~\ref{thm:main} and we point out that the equivariance holds since $S(\FF)^{d-k}$ acts on the domain by inverse-product diagonal action. As explained in \eqref{eq:equiv-map-hocolim}, the maps $\Phi_i$ are glued together to define an $S(\FF)^{d-k}$-equivariant map $\Phi \colon X \to (\FF^{r+1})^{d-k}$. If $0 \in \im(\Phi)$, then there exists a face $\{i_1 < \dots < i_m\} \in P(M)$, real numbers $t_1, \dots, t_\ell > 0$ with $\sum_{\ell = 1}^m t_\ell = 1$, and $(z_\ell)_{\ell = 1}^m \in (C_{d-k}(\FF))^m$ such that $\{i_\ell\}_{\ell = 1}^m$ and $(z_\ell)_{\ell = 1}^m$ satisfy condition $(c_1)$ and
\begin{align*}
    \sum_{\ell = 1}^m \Big(t_\ell z_{j, \ell}, t_\ell z_{j, \ell} \phi(F_{i_\ell}) \Big)_{j=1}^{d-k} = 0 \in (\FF^{r+1})^{d-k}.
\end{align*}
Since $\phi$ is assumed to model $(d-k, M,\FF)$-dependencies of $\F$ (recall Definition~\ref{def:model-M,k-dep}), there exists some real numbers $r_1, \dots, r_m \ge 0$ and points $q_1 \in F_{i_1}, \dots, q_\ell \in F_{i_\ell}$ such that the affine $\FF$-dependencies
\begin{equation*}
    \sum_{\ell=1}^m t_\ell r_\ell z_{j,\ell} = 0 \in \FF,~ \sum_{\ell=1}^m t_\ell r_\ell z_{j,\ell} q_\ell = 0 \in \FF^d, \hspace{3mm} \text{for}~ j=1, \dots, d-k,
\end{equation*}
are not all trivial, i.e., some $t_\ell r_\ell z_{j,\ell} > 0$. However, this contradicts the condition ($c_1$). In conclusion, $0 \notin \im(\Phi)$, as claimed.

We may now complete the proof. There is a composition of $\Z_2^{d-k}$-maps
\begin{equation*}
        \sk_{\rank(M)-1}(\stiefel{d-k}{d+1}) \xrightarrow{~s~}  E \xrightarrow{~p_2~} X \xrightarrow{~\Phi~} (\FF^{r+1})^{d-k} \setminus \{0\} \xrightarrow{~~~} S((\FF^{r+1})^{d-k}),
\end{equation*}
where the last arrow is the equivariant retraction $y \mapsto y/\|y\|$. However, this is in contradiction with Theorem~\ref{thm:equiv-map}, since $\rank(M) - 1 = \dim_\R(\FF) \cdot (d-k)(r+1)$.

\section{Proof of a Dol'nikov-type theorem}
\label{sec:dolnikov-matroid}

In this section we prove Theorem~\ref{thm:dolnikov-matroids} using our main result; the proof is a modification of the argument in the nonmatroidal case~\cite[Lem.~4.1]{mcginnis2026necessary}.

Given a family $\F = \F_1 \cup \dots \cup \F_{r+1}$ of convex sets in $\FF^d$ and a matroid $M$ on $\F$, by Theorem~\ref{thm:main} it is enough to show that there exists a function $\phi \colon \F \to \FF^r$ that models $(d-k,M,\FF)$-dependencies of $\F$ (recall Definition~\ref{def:model-M,k-dep}).
For $\FF$-affinely independent points $p_1, \dots, p_{r+1} \in \FF^r$ we will demonstrate that $\phi$ that maps the $\F_i$ constantly to $p_i$, for each $i=1, \dots, r+1$, satisfies the property.
To this end, let $\I$ be an independent set in $M$ and suppose that
\begin{equation*} 
    \sum_{F \in \I} a_{j,F} = 0 \in \FF,~ \sum_{F \in \I} a_{j,F} \phi(F) = 0 \in \FF^r, \hspace{3mm} \text{for}~ j=1, \dots, d-k,
\end{equation*}
are any $d-k$ affine $\FF$-dependencies which are not all trivial (i.e., some $a_{j,F} \neq 0$). Independence of points $p_1, \dots, p_{r+1}$ yields that
\begin{equation*}
    \sum_{F \in \I \cap \F_i} (a_{1,F}, \dots, a_{d-k,F}) = 0 \in \FF^{d-k}, \hspace{3mm} \text{for}~ i=1, \dots, r+1.
\end{equation*}
Since the $a_{j,F}$ are not all zero, there exists at least one $i_0 \in \{1, \dots, r+1\}$ for which the above sum does not consists exclusively of zero $(d-k)$-vectors. Interpreting this sum as claiming that the origin is in the convex hull of the $(d-k)$-vectors, by Carath\'eodory's theorem~\cite{caratheodory1911variabilitatsbereich}, there exists a subset $J \subseteq \I \cap \F_{i_0}$ of size at most $\dim_\R(\FF^{d-k}) +1 = \dim_\R(\FF)\cdot (d-k) + 1$ and positive real numbers $t_F$, for each $F \in \J$, such that
\begin{equation*}
    \sum_{F \in \J} t_F \cdot (a_{1,F}, \dots, a_{d-k,F}) = 0 \in \FF^{d-k}.
\end{equation*}
Given that $\J$ is an independent set in $M[\F_{i_0}]$, from the assumption of the theorem we conclude that there exists a point $q \in \bigcap \J$. Therefore, we obtain
\begin{equation*} 
    \sum_{F \in \I} r_F a_{j,F} = 0 \in \FF,~ \sum_{F \in \I} r_Fa_{j,F} q_F = 0 \in \FF^d, \hspace{3mm} \text{for}~ j=1, \dots, d-k,
\end{equation*}
are affine $\FF$-dependencies that are not all trivial if we set $r_F \coloneqq t_F$ and $q_F \coloneqq q$, for $F \in \J$, and $r_F\coloneqq 0$ and $q_F \in \F$ any point, otherwise. This completes the proof of the theorem.

\section{Acknowledgements}

I would like to thank Daniel McGinnis for useful discussions and comments, as well as Florian Frick, Andreas Holmsen, and Pablo Sober\'on for their feedback on earlier drafts of the paper.

%\printbibliography

\bibliographystyle{plain}
\bibliography{references.bib}
\end{document}